\documentclass{amsart}
\usepackage{epsfig}
\usepackage{amsmath}
\usepackage{amssymb}
\usepackage{amscd}
\usepackage{graphicx}
\usepackage{color}
\usepackage{verbatim}

\newtheorem{thm}{Theorem}[section]
\newtheorem{lem}[thm]{Lemma}
\newtheorem{prop}[thm]{Proposition}

\newtheorem{cor}[thm]{Corollary}
\newtheorem{defn}[thm]{Definition}

\theoremstyle{remark}
\newtheorem{rem}[thm]{Remark}
\newtheorem{que}[thm]{Question}

\newtheorem{exam}[thm]{Example}

\def \N {\mathbb N}

\def \K {\mathcal K}

\def \Z {\mathbb Z}
\def \R {\mathcal R}

\def \Q {\mathcal Q}

\def \U {\mathcal U}
\def \M {\mathcal M}
\def \P {\mathcal P}

\def \F {\mathcal F}

\def \U {\mathcal U}
\def \V {\mathcal V}
\def \htop{h_{\mathsf{top}}}
\def \Htop{H_{\mathsf{top}}}

\def \sq {sequence}

\def \tl {topological}

\newcommand{\card}[1]{\left|{#1}\right|}
\newcommand{\sd}{\triangle}
\numberwithin{equation}{section}

\author{Tomasz Downarowicz, Bartosz Frej and Pierre-Paul Romagnoli}

\title[Shearer's inequality and Infimum Rule]{Shearer's inequality and Infimum Rule for Shannon entropy and topological entropy}

\thanks{The research of the first named two authors is supported by the NCN (National Science Center, Poland) grant 2013/08/A/ST1/00275. The third named author acknowledges the support of 
Programa Basal PFB 03, CMM, Universidad de Chile. Part of the research was conducted during the 
visit of the first named author at CMM, Universidad de Chile.}

\begin{document}

\begin{abstract}
We review subbadditivity properties of Shannon entropy, in particular, from the Shearer's inequality
we derive the ``infimum rule'' for actions of amenable groups. We briefly discuss applicability of the ``infimum formula'' to actions of other groups. Then we pass to topological entropy of a cover. We prove
Shearer's inequality for disjoint covers and give counterexamples otherwise. We also prove that, for actions of amenable groups, the supremum over all open covers of the ``infimum fomula'' gives correct 
value of \tl\ entropy.
\end{abstract}



\maketitle

\centerline{February 24, 2015}

\section{Introduction}
This note is devoted to properties of entropy, both measure-theoretic and \tl, treated as a function defined on subsets of the acting group. One such property, called subadditivity, is popularly known and used. It implies, in particular, that when evaluating the dynamical entropy (of a partition or of an open cover) in an action of $\Z$, we can exchangeably apply $\limsup$, $\liminf$, $\lim$ or $\inf$, of the terms $\frac1n H(\{1,2,\dots,n\})$ (where $H(\{1,2,\dots,n\})$ is appropriately understood), simply because the \sq\ $\frac1n H(\{1,2,\dots,n\})$ converges to its infimum. Similar statement holds for actions of countable amenable groups, with $\{1,2,\dots,n\}$ replaced by elements $F_n$ of a F\o lner \sq\ (and $\frac1n$ replaced by $\frac1{|F_n|}$). But the measure-theoretic entropy fulfills a stronger property, called strong subadditivity, which implies that the same value will be obtained when taking infimum of the terms $\frac1{|F|}H(F)$ over all finite subsets $F$ of the acting group. Notice that this ``infimum rule'' allows to define (and evaluate) the entropy of a measure-preserving action of an amenable group without referring to any F\o lner sequence. In other words, the simple form $\inf_F \frac1{|F|}H(F)$ does not depend upon amenability of the group and can be used as a definition of entropy of processes in actions of arbitrary groups. A natural question arises: is this definition good (i.e., does it fulfill desirable postulates one expects from a reasonable notion of entropy)? We will briefly discuss this question in the section \emph{Beyond amenability}.

For the moment, it is not known whether a similar ``infimum rule'' applies to \tl\ entropy of an open cover (even for the actions of $\Z$). As one of our examples shows, the corresponding function is not strongly subadditive, but this does not determine that the infimum rule fails. In fact, this rule is implied by a property lying between subadditivity and strong subadditivity, called Shearer's inequality. We will show that if the open cover consists of disjoint sets, Shearer's inequality holds, hence, for actions of countable amenable groups, the infimum rule does work. For non-disjoint covers, we show that Shearer's inequality fails. Moreover, we and give an example of a $\Z_3$-action in which the infimum does not hold. We do not have an analogous example for $\Z$---this open problem seems to be difficult. On the other hand, we prove that for amenable groups  the supremum over all open covers of the ``infimum formula'' does yield the correct value of \tl\ entropy of the action.
\smallskip

We remark that all, presented in this note, results concerning measure-theoretic entropy are well known. Novel are only the results concerning \tl\ entropy (in particular the examples). 


\section{Subadditivity and related notions}
Let $G$ be an abstract set and let $\F(G)$ be the collection of all nonempty finite subsets of $G$.
A \emph{$k$-cover} of a set $F\in\F(G)$ is a family $\{K_1,K_2,\dots,K_r\}$ of elements of $\F(G)$
(with possible repetitions) such that each element of $F$ belongs to $K_i$ for at least $k$ indices $i\in\{1,2,\dots,r\}$. With slight abuse of precision, we will say ``belongs to at least $k$ elements of $\K$'' and the sums and products over $i\in\{1,2,\dots,r\}$ will be indexed by $K\in\K$ (we must remember that repeated terms $K$ are counted separately).

Let $H$ be a nonnegative real function with domain $\F(G)$. 
\begin{defn}
We say that:
\begin{enumerate}
	\item[(M)] $H$ is \emph{monotone} if $F\subset F'$ implies $H(F)\leq H(F')$;
	\item[(S)] $H$ is \emph{subadditive} if for any $F,F'$ it holds that 
	\[
	H(F\cup F')\leq H(F) + H(F');
	\]
	\item[(Sh)] $H$ satisfies \emph{Shearer's inequality} if for any $F$ and any 
	$k$-cover $\K$ of $F$,
	\[
	H(F) \leq \frac1k \sum_{K\in\K} H(K);
	\]
	\item[(SS)] $H$ is \emph{strongly subadditive} if for any $F,F'\subset G$,
	\[
	H(F\cup F')\leq H(F) + H(F') - H(F\cap F').
	\]
\end{enumerate}
\end{defn}
Let us define the \emph{conditional value of $H$} by the formula
\begin{equation}	\label{conditional}
H(F|F')=H(F\cup F') - H(F').
\end{equation}

Using the conditional value we introduce two more notions:
\begin{enumerate}
	\item[(MC)] $H$ is \emph{monotone wrt. the condition} if $F'\subset F''$ implies $H(F|F')\geq H(F|F'')$,
	\item[(CS)] $H$ is \emph{conditionally subadditive} if for any $F,F',F''$ it holds that 
	\[
	H(F\cup F'|F'')\leq H(F|F'') + H(F'|F'').
	\]
\end{enumerate}
\begin{rem}
If we extend $H$ by setting $H(\emptyset) = 0$, then $H(F|\emptyset) = H(F)$ for
every $F\in\F(G)$ and monotonicity wrt. the condition includes that $H(F|F')\le H(F)$.
\end{rem}

\begin{lem}
If $H$ is monotone then the conditions (SS), (MC) and (CS) are equivalent.
\end{lem}

\begin{proof} (SS) $\implies$ (MC): For $F'\subset F''$,
\begin{multline*}
H(F|F'') = H(F \cup F'') - H(F'') = H(F \cup F' \cup F'') - H(F'') \\
	\leq H(F \cup F') + H(F'') - H((F \cup F') \cap F'') - H(F'')\\
	\leq H(F\cup F') - H(F') = H(F|F').
\end{multline*}
(MC) $\implies$ (CS) (monotonicity is not used):
\begin{multline*}
H(F\cup F'|F'') = H(F\cup F' \cup F'') - H(F'') + H(F'\cup F'') - H(F'\cup F'') \\
= H(F|F'\cup F'') + H(F'|F'')\leq H(F|F'') + H(F'|F'').
\end{multline*}
(CS) $\implies$ (SS) (monotonicity is not used):
\begin{multline*}
H(F\cup F') =  H(F\cup F' \cup (F\cap F')) = H(F\cup F'|F\cap F') + H(F\cap F')\\
\leq H(F|F\cap F') + H(F'|F\cap F') + H(F\cap F') = \\
H(F) - H(F\cap F') + H(F') - H(F\cap F') + H(F\cap F') = H(F) + H(F') - H(F\cap F').
\end{multline*}
\end{proof}

\begin{prop}	\label{SS_then_SI}
If $H$ is monotone then (SS) $\Rightarrow$ (Sh) $\Rightarrow$ (S) and none of the implications may be reversed.
\end{prop}

\begin{proof}
Assume that $H$ is strongly subadditive. Let $F=\{f_1,f_2,...,f_m\}$ and let $\K$ be a $k$-cover of $F$. Using (\ref{conditional}) we can write
\[
H(F) = H(\{f_1\}) + H(\{f_2\}|\{f_1\}) + H(\{f_3\}|\{f_1,f_2\}) + ... + H(\{f_m\}|\{f_1,...,f_{m-1}\})
\]
and similarly, for each $K\in\K$,
\[
H(K)= \sum_{\{j:f_j\in K\}} H\left(\{f_j\} | \{f_i\in K : i<j\}\right).
\]
By the preceding lemma, $H$ fulfills (MC), hence $H(\{f_j\} | \{f_i\in K : i<j\})\geq H(\{f_j\} | \{f_1,...,f_{j-1}\})$.
Since each $f_j$ belongs to at least $k$ elements of $\K$, summing over $\K$, we obtain
\[
\sum_{K\in\K}H(K) \geq \sum_{j=1}^m kH(\{f_j\} | \{f_1,...,f_{j-1}\}) = kH(F).
\]

For the proof of (Sh) $\Rightarrow$ (S) note that $\{F, F'\setminus F\}$ is a $1$-cover of $F\cup F'$, hence $H(F\cup F') \leq H(F) + H(F'\setminus F) \leq H(F) + H(F')$. 
\medskip

For counterexamples we ask the reader to see section \ref{top_entropy}. In Example~\ref{golden_mean}  we show that topological entropy of the standard time-zero partition (which is also a cover) in the golden mean shift is not strongly subadditive. On the other hand, it satisfies Shearer's inequality, because the cover is disjoint (see Proposition \ref{SI_for_disjoint}).

In example \ref{IR_and_SI_fail} we present a $\Z_3$-action such that topological entropy of a certain 
cover does not satisfy Shearer's inequality (it does not even satisfy the infimum rule, see below). 
On the other hand, it is known that topological entropy of a cover is subadditive. 
\end{proof}


\section{The infimum rule}
Let $G$ be an amenable group and let $(F_n)$ be a selected F{\o}lner sequence. By $\card{F}$ we will denote the cardinality of $F$.
\begin{defn}
We will say that a nonnegative function $H$ on $\F(G)$ satisfies the \emph{infimum rule} if
\[
\limsup_{n\to\infty} \frac1{\card{F_n}} H(F_n) = \inf_{F\in\F(G)} \frac1{\card{F}}H(F).
\]
\end{defn}
\begin{defn}
We say that $H$ is \emph{$G$-invariant} if for any $g\in G$ it holds that $H(Fg)=H(F)$.
\end{defn}

\begin{prop}	\label{SI_then_IR}
If a nonnegative and $G$-invariant function $H$ on $\F(G)$ satisfies Shearer's inequality then it also obeys the infimum rule.
\end{prop}

\begin{proof}
Clearly, $\limsup_{n\to\infty} \frac1{\card{F_n}} H(F_n) \geq \inf_{F\in\F(G)}\frac1{|F|} H(F)$. 
\medskip

For the converse inequality, fix an $F\in \F(G)$ and $\epsilon>0$. 
For $n$ large enough the F{\o}lner set $F_n$ is $(F^{-1},\epsilon)$-invariant, i.e., it satisfies $\frac{\card{F_n\sd F^{-1}F_n}}{\card{F_n}}< \epsilon$. The family
$$
\K=\{Fg : g\in G, Fg\cap F_n\neq\emptyset\}
$$
is a $k$-cover of $F_n$ with $k = |F|$ (for $g\neq g'$, we treat the sets $Kg$ and $Kg'$ as different
elements of the $k$-cover, even if they are equal as sets). Indeed, for $f\in F_n$, the condition $f\in Fg$ can be written as $g\in F^{-1}f$ so, it is fulfilled for exactly $k$ elements $g$. By the same calculation, $Fg\cap F_n\neq\emptyset$ if and only if $g\in F^{-1}F_n$, so the cardinality of $\K$ equals that 
of $F^{-1}F_n$, i.e., it is not more than $|F_n|(1+\epsilon)$.
By invariance of $H$, $H(K) = H(F)$ for every $K\in\K$.
The Shearer's inequality now reads
$$
H(F_n)\le \frac1k \sum_{K\in\K}H(K) \le \frac1{|F|} |F_n|(1+\epsilon)H(F),
$$
which, after dividing by $|F_n|$ and passing with $n$ to infinity, ends the proof.
\end{proof}

\section{Shannon entropy}
Let $(X,\Sigma,\mu)$ be a probability space. A \emph{partition} of $X$ is a finite collection $\P$ 
of pairwise disjoint measurable sets such that $\bigcup_{A\in\P} A = X$. By a \emph{join} (or a \emph{common refinement}) of partitions $\P$ and $\P'$ we mean the partition $\P\vee\P'=\{A\cap B: A\in\P,\ B\in\P'\}$.  Now, let $G$ be an amenable group acting on $X$ via measurable maps, which preserve the measure $\mu$. If $F$ is a finite subset of $G$ we write $\P^F$ for the common refinement $\bigvee_{g\in F} g^{-1}\P$, where $g^{-1}\P=\{g^{-1}A: A\in \P\}$ . 
Recall that the \emph{Shannon entropy} of a partition $\P$ is defined by
\[
H_\mu(\P) = -\sum_{A\in\P} \mu(A)\log\mu(A)
\]
and \emph{entropy of the action of $G$ with respect to a partition} $\P$ is defined as
\[
h_\mu(G,\P) = \limsup_{n\to\infty} \frac1{\card{F_n}}H_\mu\left(\P^{F_n}\right)
\]
($(F_n)$ is a F\o lner \sq\ in $G$).
\begin{prop}
The function $H(F)=H_\mu(\P^F)$ is
\begin{enumerate}
	\item nonnegative,
	\item $G$-invariant,
	\item monotone,
	\item strongly subadditive.
\end{enumerate}
\end{prop}
These statements follow from standard properties of the entropy of a partition---for the proofs we refer the reader to any 
handbook on ergodic theory, ((SS) is usually replaced by (MC) or (CS), see e.g. \cite[(1.6.7) and (1.6.9)]{D}). 
%
%
In view of these facts, the following theorem follows from Propositions \ref{SS_then_SI} and \ref{SI_then_IR}.
\begin{thm}\label{infr}
Measure-theoretic entropy obeys the infimum rule, i.e. 
\[
h_\mu(G,\P) = \inf_{F\in\F(G)} \frac1{\card{F}}H_\mu(\P^F).
\]
for every $G$-invariant measure $\mu$ and every partition $\P$.
\end{thm}
\begin{rem}
The above formula can be found, e.g. in \cite{S1}, where it is attributed to Kolmogorov.
\end{rem}

\section{Beyond amenability}
The ``mindblowingly'' simple formula $\inf_{F\in\F(G)} \frac1{\card{F}}H_\mu(\P^F)$ can be applied to processes under actions of any countable groups. (It can be applied to uncountable groups as well, however, it will typically 
yield zero; such is the case of flows.) To distinguish from other existing notions, we will denote it by $h_\mu^*(G,\P)$. How good is this formula for countable non-amenable groups?
The answer depends on the properties we expect from a good notion of dynamical entropy.

The notion $h_\mu^*(G,\P)$ has the following advantages:
\begin{itemize}
	\item It is completely universal, can be defined for arbitrary groups.
	\item It is extremely simple, requires no details of the group (for instance in amenable groups it is 		
	formulated without referring to any F\o lner \sq).
	\item It satisfies \emph{the Bernoulli shifts postulate}: Bernoulli shifts have ``full'' entropy (equal 		to the Shannon entropy of the independent generator). 
	\item It has a very convincing interpretation for other processes (entropy is lost in finite-dimen\-sional dependencies and \emph{all such losses matter}).
\end{itemize}

Disadvantages can be detected by examining the action of the free group $F_2$ with two generators, and they
include:
\begin{itemize}
	\item It fails \emph{the factors postulate}: it can increase when passing to a factor.
	\item It fails \emph{the invariance postulate}: it can change with change of a generator (hence is not an isomorphism invariant). 
\end{itemize}

Before the examples, we recall the notions of the shift action and of a subshift.
Suppose $G$ is a group and $\Lambda$ is a finite set with the discrete topology. By the \emph{full shift} 
we understand the set $\Lambda^G$ (whose elements are $x=(x_g)_{g\in G}$) equipped with the product 
topology, on which $G$ acts by shifts: $(gx)_h = x_{hg}$. A \emph{subshift} is any closed $G$-invariant subset $X\subset\Lambda^G$. The full shift, as well as any subshift, admits a natural partition $\P_\Lambda=\{[a]:a\in\Lambda\}$ by closed-and-open cylinder sets determined by the symbol ``at zero'': $[a] = \{x:x_e = a\}$ ($e$ is the unity of $G$). We call it the \emph{time-zero partition} (or \emph{time-zero cover}, depending on the context). The term \emph{Bernoulli measure} is synonymous with a product measure $\nu^G$ on $\Lambda^G$, where $\nu$ is a probability measure on $\Lambda$.

\begin{exam}
Let $F_2$ denote the free group with two generators $a$ and $b$ and unity $e$, and consider $X=\{-1,1\}^{F_2}$ with 
the shift action, the Bernoulli $(\frac12,\frac12)$-measure, and the zero-coordinate partition $\P = \{[-1],[1]\}$. Clearly, $H(\P)=\log 2$ and $h_\mu^*(F_2,\P)=\log 2$. 
Next, consider the mapping $\psi:X\to\{-1,1\}\times\{-1,1\}$ given by 
$$
\psi(x) = (x(e)x(a), x(e)x(b))
$$
and the associated four-element partition $\mathcal R$. It is not hard
to see that the process generated by $\mathcal R$ is the $(\frac14,\frac14,\frac14,\frac14)$-Bernoulli shift: the one-dimensional distributions are independent. So, $H_\mu(\mathcal R)=\log 4$ and so equals $h_\mu^*(F_2,\mathcal R)$. On the other hand, the process generated by $\mathcal R$ is clearly a factor of that generated by $\P$. \footnote{The above example shows that the failure of the factors postulate is inevitable for any entropy notion satisfying the Bernoulli shifts postulate (in particular, for \emph{sofic entropy} \cite{LB}). For this reason it is commonly agreed to give up the factors postulate in the search for a universal (i.e., valid for a range of acting groups) notion of entropy.}
 
Now let $E = \{e,a,b\}\subset F_2$ and consider $\Q=\P^E$. Clearly, this partition is another generator
of the process generated by $\P$ (the generated processes are isomorphic). For any finite set $F\subset F_2$ we have $H(\Q^F)=H(\P^{EF}) = |EF|\log 2$. However, the ratio $\frac{|FE|}{\card{F}}$ does not drop below $2$ (and can be arbitrarily close to $2$). Hence $h_\mu^*(F^2,\Q)=2h_\mu^*(F^2,\P)=\log 4$. \footnote{The strength of the notion of \emph{sofic entropy} is that it behaves better in this aspect; it does not depend on the partition as long as it generates the whole process. So, sofic entropy can be viewed as a parameter associated to measure-preserving actions, and becomes an isomorphism invariant. On the other hand, sofic entropy has its disadvantages: it has a very complicated definition, and it applies to actions of sofic groups only.}
\end{exam}

So, either we accept $h_\mu^*(G,\P)$ as a parameter associated with a concrete \emph{process}, maintaining its simplicity and interpretation, or we try to force it to become an isomorphism invariant. As an attempt in this direction we propose two invariants, both equal to $h_\mu^*(G,\P)$ for actions of amenable groups. Unfortunately, we are unable to verify whether these new notions fulfill the Bernoulli shift condition in a more general case.

\begin{defn}\label{att}
$$
h_\mu^{**}(X,G) = \inf\{h_\mu^*(G,\P):\ \P \text{ is a generator}\},
$$
$$
h_\mu^{***}(X,G) = \inf\{H_\mu(\P):\ \P \text{ is a generator}\}.
$$
\end{defn}

Note that the latter notion has nothing to do with $h_\mu^*(G,\P)$, we were driven to it just by analogy to 
$h_\mu^{**}(X,G)$. For actions of amenable groups we know that $h_\mu^{**}(X,G)$ equals $h_\mu(X,G)$. We also 
have $h_\mu^{***}(X,G) = h_\mu(X,G)$ (see \cite[Corollary 2.7]{ST}). 

In the general case it is obvious that $h^{**}(X,G)\le h^{***}(X,G)$. B. Seward \cite{S2} can prove the opposite inequality for free actions
(i.e., such that for $g\neq e$, the set of points fixed by $g$ has measure $0$). As we already mentioned, we still do not know whether any of these notions satisfies the Bernoulli shifts postulate.\footnote{B. Weiss (\cite{W}) can prove the Bernoulli shifts postulate for $h^{***}(X,G)$ in actions of sofic groups, so $h^{***}(X,G)$ becomes a serious competition for sofic entropy. It is unknown whether these two notions coincide for actions of sofic groups.}


\section{Topological entropy}	\label{top_entropy}
In the present section we enter the world of topological dynamical systems. We assume that $X$ is a compact metric space and $G$ acts by homeomorphisms on $X$. Similarly to the measure-theoretic case, for an open cover $\U$ and a finite $F\subset \N$ we write $\U^F$ for the refinement $\bigvee_{f\in F} f^{-1}\U$, where $f^{-1}\U=\{f^{-1}U: U\in \U\}$. We recall that $N(\U)$ is the smallest cardinality of a subcover chosen from a cover $\U$ and that topological entropy of a cover $\U$ is defined by $\Htop(\U)=\log(N(\U))$. Topological entropy of the action
is defined in two steps:
\begin{gather*}
\htop(T,\U) = \limsup_{n\to\infty} \frac1{|F_n|} \Htop(\U^{F_n}),\\
\htop(T) = \sup_{\U} \htop(T,\U),
\end{gather*}
where $(F_n)$ is a F\o lner \sq\ and the supremum is taken over all open covers of~$X$. 

We want to study the nonnegative function on $\F(G)$ obtained by fixing an open cover $\U$ and abbreviating $\Htop(\U^F)$ as $\Htop(F)$. It is obvious that $\Htop$ is monotone and $G$-invariant. It is also commonly known (and easily verified) that this function is subadditive (see e.g. \cite[(6.3.8)]{D}). The natural next step is the verification of Shearer's inequality. We begin with a discrete version of \cite[Thm. 2]{bato}. The proof is almost literally copied.

\begin{lem}\label{baba}
Let $\mathfrak X$ be a subset of $\Lambda^n$, where $\Lambda$ is a finite set and $n\in\N$. Let $\K$ be a $k$-cover
of the set of coordinates $\{1,2,\dots,n\}$ (the elements of $\K$ are nonempty subsets of $\{1,2,\dots,n\}$, we admit repeated elements in $\K$, and each coordinate belongs to at least $k$ elements of $\K$). For $K\in\K$ let $\mathfrak X_K$ denote the projection of $\mathfrak X$ onto the coordinates belonging to $K$. Then
$$
|\mathfrak X|\le \prod_{K\in\K} |\mathfrak X_K|^{\frac1k}.
$$
\end{lem}

\begin{proof}
For $n=1$ the statement is obvious: each $K\in\K$ equals $\{1\}$, each $\mathfrak X_K$ equals $\mathfrak X$ and the cardinality 
of $\K$ is at least $k$. 
We proceed by induction. Consider an $n\ge 2$ and suppose the statement holds for subsets of $\Lambda^{n-1}$. The set $\mathfrak X\subset\Lambda^n$ splits into disjoint sets $\mathfrak X^a$, $a\in\Lambda$, depending
on the value at the last coordinate:
$$
\mathfrak X^a = \{(x_1,x_2,\dots,x_n)\in \mathfrak X: x_n=a\}.
$$

For $K\in\K$ let $K^\circ = K\setminus\{n\}$ and let $\K^\circ = \{K^\circ:K\in\K\}$. 
Clearly, $\K^\circ$ is a $k$-cover of $\{1,2,\dots,n-1\}$. For every $a\in\Lambda$, $\mathfrak X^a$ can be 
viewed as a subset of $\Lambda^{n-1}$ (with the symbol $a$ appended to each element), hence, by 
the inductive assumption, we have
$$
|\mathfrak X^a|\le \prod_{K\in\K} |\mathfrak X^a_{K^\circ}|^{\frac1k}. 
$$
Because in $\mathfrak X^a$ the symbol at the last coordinate is determined, we have $|\mathfrak X^a_{K^\circ}|=|\mathfrak X^a_K|$ for every $K$, and we get
$$
|\mathfrak X^a|\le \prod_{K\in\K} |\mathfrak X^a_K|^{\frac1k}.
$$

Further, 
$$
|\mathfrak X| = \sum_{a\in\Lambda}|\mathfrak X^a|, \text{ \ and \ }|\mathfrak X_K| = \sum_{a\in\Lambda}|\mathfrak X^a_K|,
$$
for every $K\in\K$ such that $n\in K$. For $K$ not containg $n$ we will apply the estimate $|\mathfrak X^a_K|\le|\mathfrak X_K|$ (regardless of $a\in\Lambda)$.
And so, we have
\begin{multline}\label{here}
|\mathfrak X| = \sum_{a\in\Lambda}|\mathfrak X^a|\le \sum_{a\in\Lambda}\prod_{K\in\K}|\mathfrak X^a_K|^{\frac1k}=\sum_{a\in\Lambda}\Bigl(\prod_{K\ni n}|\mathfrak X^a_K|^{\frac1k}\cdot
\prod_{K\not\ni n}|\mathfrak X^a_K|^{\frac1k}\Bigr)\\
\le \prod_{K\not\ni n}|\mathfrak X_K|^{\frac1k} \cdot \sum_{a\in\Lambda}\prod_{K\ni n}|\mathfrak X^a_K|^{\frac1k}.
\end{multline}

For each $K$ containing $n$, on $\Lambda$ we define a function $f_K$, by $f_K(a) = |\mathfrak X^a_K|^{\frac1k}$, and then we 
apply the generalized H\"older inequality:
$$
\Bigl\|\prod_{K\ni n}f_K\Bigr\|_p \le \prod_{K\ni n}\Bigl\|f_K\Bigr\|_k,
$$
where $\frac1p=\sum_{K\ni n}\frac1k$. Because $\K$ is a $k$-cover, this sum has at least $k$ terms, hence $p\le 1$
(if $p<1$, formally, $\|\cdot\|_p$ is not a norm, but it does not matter). Since for a fixed finite-dimensional 
vector $f$, the term $\|f\|_p$ is a decreasing function of $p>0$, the above inequality holds also for $p=1$, and then it reads:
$$
\sum_{a\in\Lambda}\prod_{K\ni n}|\mathfrak X^a_K|^{\frac1k}\le \prod_{K\ni n}\Bigl(\sum_{a\in\Lambda}|\mathfrak X^a_K|\Bigr)^{\frac1k} = 
\prod_{K\ni n}|\mathfrak X_K|^{\frac1k}.
$$
Plugging this into \eqref{here} we end the proof.
\end{proof}

\begin{cor}	\label{SI_for_disjoint}
If $\U$ consists of pairwise disjoint sets then the corresponding function $\Htop$ on $\F(G)$ fulfills Shearer's inequality.
\end{cor}

\begin{proof} Because the cover $\U$ is by disjoint sets, for each $F\in\F(G)$ we have
$N(\U^F) = |\U^F|$, where $\U^F$ is rid of empty elements. Labeling $\U$ by elements 
of a finite alphabet $\Lambda$ (of the same cardinality as $\U$), $\U^F$ can be identified 
with a set $\mathfrak X\subset\Lambda^F$, while for any $K\subset F$ we have $\U^K = \mathfrak X_K$.
Now, $\Htop(F)=\log|\mathfrak X|$ and $\Htop(K) = \log|\mathfrak X_K|$ and the assertion follows directly 
from Lemma \ref{baba}.
\end{proof}

\begin{cor}
If $(X,G)$ is a subshift and $\U=\P_\Lambda$ then the infimum rule holds 
for $\Htop$, i.e.,
\[
\htop(T)=\htop(T,\U) = \inf_{F\in\F(G)} \frac1{\card{F}} \Htop(\U^F).
\]
\end{cor}
\begin{proof}
Follows immediately from Propositions \ref{SI_for_disjoint} and \ref{SI_then_IR}.
\end{proof}

The following example shows that \tl\ entropy is not in general strongly subadditive, even for subshifts. 
To make matters worse, in another example we will present a system and a (non-disjoint) cover such that both Shearer's inequality and infimum rule fail.

\begin{exam}	\label{golden_mean}
The golden mean shift is a subshift $X\subset \{0,1\}^\Z$ consisting of all sequences in which block $11$ does not occur. Let $\U=\P_\Lambda$ be the two-element time-zero cover. It was shown already that Shearer's inequality holds, but it can easily be checked that
\begin{multline*}
\Htop(\{-1,0,1\}) + \Htop(\{0\}) = \log 5 + \log 2 \geq \log 3 + \log 3 \\
= \Htop(\{-1,0\}) + \Htop(\{0,1\}),
\end{multline*}
i.e., strong subadditivity does not hold.
\end{exam}

The following example looks innocent, but it took us a lot of effort to find it.

\begin{exam}	\label{IR_and_SI_fail}
Let $\Lambda=\{a,b,c\}$. We consider the action of the group $\Z_3=\Z/3\Z = \{0,1,2\}$, by shifts, on a space $X\subset \Lambda^{\Z_3}$ (i.e., a subshift) defined by
\[
X=\{(a,a,a), \ (b,b,b), \ (c,c,c), \ (a,b,c), \ (b,c,a), \ (c,a,b)\}.
\]
(In fact, $(c,c,c)$ can be dropped and the example with 5 elements will still work.) It is clear that $X$ is closed and $\Z_3$-invariant. The collection $\V=\{\{a,b\},\{b,c\},\{a,c\}\}$ is a (non-disjoint) cover of $\Lambda$. Let $\bar\V = \{\bar V:V\in\V\}$, where $\bar V= V\times\Lambda\times\Lambda$. Clearly, $\bar\V$ is a cover of $\Lambda^{\Z_3}$ hence also of $X$. The elements of $\bar\V^{\Z_3}$ have the form $V_1\times V_2\times V_3$, where $V_1,V_2,V_3$ are (not necessarily distinct) elements of $\V$. It is easy to check that $X$ admits a subcover of $\bar\V^{\Z_3}$ consisting of three sets, namely
\[
\{a,b\}\times\{a,b\}\times\{a,b\},\ \ \ \{a,c\}\times\{b,c\}\times\{a,c\},\ \ \ \{b,c\}\times\{a,c\}\times\{a,b\}.
\]
On the other hand, there is no subcover with two elements (one set of the form $V_1\times V_2\times V_3$ may contain at most two elements of $X$, because in any three of them, on some coordinate there
appear all three letters). Thus $N(\bar\V^{\Z_3})=3$ and 
$$
\htop(\Z_3,\bar\V)=\frac1{|\Z_3|}\Htop(\bar\V^{\Z_3}) = \frac13\log3
$$ 
(in any finite group $G$, for any F\o lner \sq, eventually $F_n=G$). However, the cover $\bar\V^{\{0,1\}}$ has a minimal subcover consisting of only two sets (for example $\{b,c\}\times\{a,c\}$ and $\{a,b\}\times\{a,b\}$), and $\Htop(\V^{\{0,1\}})=\log2$. Since $\frac12\log 2<\frac13\log 3$, the infimum rule does not hold. Consequently, by Proposition \ref{SI_then_IR} (or by direct verification for the 2-cover of $\Z_3$ by $\{0,1\}, \{1,2\}, \{0,2\}$), Shearer's inequality fails as well.
\end{exam}

\begin{rem}
Shearer's inequality depends only vaguely on the acting group. The example (as a counterexample for Shearer's inequality) can be easily adapted to the action of any group with at least 3 elements, in particular of $\Z$. The infimum rule depends more heavily on the acting 
group (its proof uses only $k$-covers obtained by shifting one set $F$ of cardinality $k$), so the following question arises:
\end{rem}

\begin{que}
Let $G$ be an infinite countable amenable group acting on a compact space $X$ and let $\U$ be an open cover of $X$. 
Does the infimum rule hold for $\Htop(F)=\Htop(\U^F)$, i.e., is it true that $\htop(G,\U)=\htop^*(G,\U)$? \end{que}

In spite of many efforts, we have not succeeded in answering this questions even for $G=\Z$. In fact, we do not even know toward which answer should we incline. Let us discuss the difficulties more extensively.

Our example for $\Z_3$ works only because the two-element set $\{0,1\}$ and its shifts form a \emph{non-splitting} 2-cover of $\Z_3$.
In $\Z$, an analogous 2-cover is splitting (i.e., it splits as a union of two $1$-covers) and it is easy to show, using plain subadditivity, that for a 
splitting $k$-cover of $G$, obtained by shifting one set $F$, the Shearer's inequality nearly holds (up to a small error) on 
large elements of the F\o lner \sq.  This suffices to prove the infimum rule. The simplest finite subset of $\Z$ whose shifts 
produce a non-splitting 3-cover is $\{0,1,3\}$, but we failed to find a counterexample for Shearer's inequality, based on 
any 3-cover. The point is that the key inequality $\frac1k\log k\le\frac1{k+1}\log(k+1)$ does not hold for $k>2$.
\medskip

Nevertheless, we are able to formulate a positive result. In its proof we benfit from the fact that the infimum rule works for the Shannon entropy and we apply the variational principle (for amenable groups it was first proved in \cite{STZ}). Let us recall: if $\M_G(X)$ is the collection of all $G$-invariant probability measures on $X$ then $\htop(G)=\sup_{\mu\in\M_G(X)} h_\mu(G)$. It suffices to take the supremum over ergodic measures.

\begin{thm}
Let $X$ be a compact metric space and $G$ a countable amenable group acting on $X$ by continuous maps. Define
\begin{gather*}
\htop^*(G,\U) = \inf_F \frac1{\card{F}} \Htop(\U^F)\\
\htop^*(G)=\sup_{\U} \htop^*(G,\U)
\end{gather*}
Then
\[
\htop^*(G)=\htop(G)
\]
\end{thm}
\begin{proof}
Clearly, $\htop^*(G)\leq \htop(G)$.

To obtain the converse inequality, consider an ergodic measure $\mu$ on $X$, a finite partition $\P=\{A_1,...,A_p\}$ of $X$, and fix an $\epsilon>0$. For any $\delta >0$ one can choose compact sets $B_i \subset A_i$, $i=1,...,p$, so that $\mu(A_i \setminus B_i)< \delta/p$. Additionally, let $B_0=X \setminus \bigcup_{i=1}^p B_i$ and define $U_i=B_0\cup B_i$, $i=1,...,p$. Then $\U=\{U_1,...,U_p\}$ is a cover with $U_i\supset A_i$, $\mu(U_i\setminus A_i)<\delta$. The family $\xi=\{B_0,B_1,...,B_p\}$ is a partition of $X$. For any $F\in\F(G)$, we have 
$$
H_\mu(\P^F) \leq H_\mu(\P^F \vee \xi^F) = H_\mu(\P^F | \xi^F)+H_\mu(\xi^F)\le |F|H_\mu(\P|\xi)+H_\mu(\xi^F)
$$ 
(we have used subadditivity of the function $H(F)=H_\mu(\P^F|\xi^F)$,
see e.g. \cite[(1.6.33)]{D}. Choosing $\delta$ appropriately small, we may demand that $H_\mu(\P | \xi) < \epsilon$, so that
\begin{equation}\label{eq1}
H_\mu(\P^F) \leq H_\mu(\xi^F) + |F|\epsilon.
\end{equation}

Let $(F_n)$ be a \emph{tempered} F\o lner \sq\ (we skip the definition, every
F\o lner \sq\ has a tempered sub\sq, the ergodic theorem holds along any tempered F\o lner \sq, see \cite{L}).
By the ergodic theorem, for almost all $x$ it holds that 
\[
\lim_{n\to\infty} \frac1{\card{F_n}}\card{\{g\in F_n:gx\in B_0\} }=\mu(B_0) < \epsilon.
\]
The same is true for $f^{-1}B_0$ ($f\in F$), which implies that for almost all $x$,
\[
\lim_{n\to\infty} \frac1{\card{F_n}}\card{\{g\in F_n:fgx\in B_0\} } < \epsilon
\]
for all $f\in F$.
Therefore, we can choose $N\in \N$ such that the set 
\[
X_{\epsilon,N}=\bigcap_{n\geq N} \bigcap_{f\in F} \left\{x\in X : \frac1{\card{F_n}}\card{\{g\in F_n: fgx\in B_0 \}} <\epsilon \right\}
\]
has positive measure. 
For $x\in X_{\epsilon,N}$ and $n\geq N$ we have
\[
\card{\left\{(g,f): f\in F,\ g\in F_n, \ fgx\in B_0 \right\}} < \card{F_n}\cdot \card{F} \cdot \epsilon,
\]
which implies
\begin{equation} \label{rect_rule}
\card{\left\{g\in F_n: \card{\left\{f\in F:fgx\in B_0\right\}} > \card{F}\sqrt{\epsilon}\right\}} \leq \card{F_n}\sqrt{\epsilon}.
\end{equation}

Consider a collection of all sets of the form $\bigcap_{f\in F} f^{-1}C_f$, where $C_f=B_0$ or $C_f=X\setminus B_0$, such that the first case occurs for at most $\card{F}\sqrt{\epsilon}$ indices $f$. Let $Y$ be the union of all such sets. Note that $Y$ is a union of elements of the partition $\xi^F$.
The ratios $\frac1{\card{F_n}} \left|\{g\in F_n: g x \in Y\}\right|$ converge to $\mu(Y)$ for $\mu$-almost every $x$. On the other hand, $g x \in Y$ is equivalent to the fact that $fgx\in B_0$ for at most $\card{F}\sqrt{\epsilon}$ elements $f\in F$. If, in addition, we pick $x \in X_{\epsilon,N}$ then (\ref{rect_rule}) implies that for every $n\ge N$,
\[
\frac1{\card{F_n}} \card{\{g\in F_n: g x \in Y\}}\geq 1-\sqrt{\epsilon}.
\]
Since $X_{\epsilon,N}$ has positive measure, it contains a point which fulfills the ergodic theorem for $Y$, implying that $\mu(Y) \geq 1-\sqrt\epsilon$.

Let $\U'$ be a subcover of $\U^F$ with minimal cardinality. 
Denoting by $\xi^F_Y$ the collection $\{B \in \xi^F: B  \cap Y \neq\emptyset\}$ (note that we can as well write $B \subset Y$, so
$\xi^F_Y$ is a partition of $Y$), we obviously have
\[
\card{\xi^F_Y} \leq \sum_{U\in \U'} \card{\{B\in\xi^F_Y: B\cap U\not=\emptyset\}}.
\]
Fix some $U\in\U'$, $U=\bigcap_{f\in F} f^{-1}U_{j_f}$, $1\leq j_f \leq p$. Consider a $B\in\xi^F_Y$ satisfying 
$B\cap U\not=\emptyset$. Then $B$ is a cylinder in $\xi^F$, $B = \bigcap_{f\in F} f^{-1}B_{k_f}$, $0\leq k_f\leq p$, and because $B$ is 
contained in $Y$, $k_f=0$ may occur for not more than $\card{F}\sqrt{\epsilon}$ indices $f\in F$. 
If $k_f\neq 0$ then it must be equal to $j_f$, because otherwise $B_{k_f}$ and $U_{j_f}$ would be
disjoint, implying $B\cap U=\emptyset$. Therefore, for given $U$ the number of sets $B\in\xi^F_Y$ intersecting $U$ is estimated
by the number of ways in which the (few) indices $0$ can be distributed over the $k_f$'s, i.e., 
\[
\card{\{B\in\xi^F_Y: B\cap U\not=\emptyset\}} \leq \sum_{i=0}^{\lfloor\card{F}\sqrt{\epsilon}\rfloor} {\card{F} \choose i}.
\]
It is well known that $\log {n\choose k} \leq n H(\frac k n)$, where
$H(\delta)$ is the binary entropy of the vector $(\delta,1-\delta)$, so we get 
\begin{equation}\label{eq2}
\log\card{\xi^F_Y} \leq \Htop(\U^F) + \card{F} H(\sqrt{\epsilon})+\log\card{F}\sqrt{\epsilon}.
\end{equation}
\smallskip

We still need to compare the partition $\xi^F_Y$ with $\xi^F$. Let $\R$ be the partition $\{Y,X\setminus Y\}$ and let $\mu_Y$ be the normalized conditional measure induced by $\mu$ on $Y$.
We have
\begin{multline*}
H_\mu(\xi^F)\le H_\mu(\xi^F\vee\R) = H_\mu(\xi^F|\R) + H(\R) \\
= \mu(Y)H_{\mu_Y}(\xi^F_Y) + \mu(X\setminus Y) H_{\mu_{X\setminus Y}}(\xi^F) + H(\R)\\
\le H_{\mu_Y}(\xi^F_Y) + \sqrt\epsilon \cdot \log |\xi^F| + H(\sqrt\epsilon).
\end{multline*}
Eventually, combining the above with \eqref{eq1} and \eqref{eq2} (and the standard estimate of entropy by means of cardinality), we get
$$
H_\mu(\P^F)\le  
\Htop(\U^F) + \card{F} H(\sqrt{\epsilon})+\log\card{F}\sqrt{\epsilon} + \sqrt\epsilon \cdot \log |\xi^F| + H(\sqrt\epsilon) + |F|\epsilon.
$$
Dividing both sides by $|F|$ and noticing that all but the first term on the right can be made arbitrarily small, regardless of $F$, 
by the choice of $\epsilon$, next taking infima over $F$ on both sides, and supremum over $\U$ on the right (which eliminates the small
error terms), we arrive to 
$$
\inf_F\frac1{|F|}H_\mu(\P^F)\le \htop^*(G).
$$
Now Theorem \ref{infr} allows to replace the left hand side by $h_\mu(G,\P)$. Since the inequality holds for any $\P$ and any ergodic $\mu$, taking appropriate suprema (and applying the variational principle), we can further replace the left hand side by $\htop(G)$, concluding the proof.
\end{proof}


\bibliographystyle{amsplain} 

\begin{thebibliography}{10}

\bibitem{bato}
  B. Bollob\' as and A. Thomason,
	\emph{Projections of bodies and hereditary properties of hypergraphs},
	Bull. London Math. Soc. {\bf 27} (1995), 417--424.

\bibitem{LB}
  L. Bowen, 
	\emph{Measure conjugacy invariants for actions of countable sofic groups}, 
	J. Amer. Math. Soc. {\bf 23} (2010), 217--245.
	
\bibitem{D}
  T. Downarowicz,
  \emph{Entropy in dynamical systems},
  Cambridge University Press, {\it New Mathematical Monographs 18}, Cambridge 2011.
	
\bibitem{L}
  E. Lindenstrauss, 
	\emph{Pointwise theorems for amenable groups}, 
	Electronic Research Announcements of AMS, {\bf 5} (1999).
	
\bibitem{S1}
  B. Seward, 
	\emph{Krieger's finite generator theorem for ergodic actions of countable groups II},
	preprint, http://arxiv.org/pdf/1501.03367.pdf
\bibitem{S2}
  B. Seward, private communication. 
	
\bibitem{ST}
  B. Seward and R. Tucker-Drob, 
	\emph{Borel structurability on the 2-shift of a countable group},
  preprint. http://arxiv.org/pdf/1402.4184.pdf
	
\bibitem{STZ}
A. Stepin and A. Tagi-Zade,
\emph{Variational characterization of topological pressure of the amenable groups of
transformations} (Russian), Dokl. Akad. Nauk SSSR {\bf 254} (1980), 545--549.
 
\bibitem{W}
  B. Weiss, private communication.
	
\end{thebibliography}

\bigskip\noindent
\address{Institute of Mathematics, Polish Academy of Science, Sniadeckich 8, 00-656 Warsaw, Poland}

\noindent
\email{downar@pwr.edu.pl}

\smallskip\noindent
\address{Departament of Mathematics, Wroclaw University of Technology, Wybrzeze Wyspianskiego 27,
50-370 Wroclaw, Poland}

\smallskip\noindent
\email{frej@pwr.edu.pl}

\smallskip\noindent
\address{Departamento de Matem\'aticas, Facultad Ciencias Exactas, Universidad Andres Bello, Rep\'ublica 252, Santiago, Chile}

\smallskip\noindent
\email{promagnoli@unab.cl}

\end{document}